\documentclass[11pt]{amsart}
\usepackage{amsmath,amssymb,amsfonts,url,mathptmx}
\numberwithin{equation}{section}
\newtheorem{Def}{Definition}[section]
\newtheorem{thm}{Theorem}[section]
\newtheorem{lem}{Lemma}[section]
\newtheorem{rem}{Remark}[section]
\newtheorem{prop}{Proposition}[section]

\begin{document}
\title[Energy Estimate]{Energy estimates for a class of semilinear elliptic equations on half Euclidean balls} \subjclass{35J15, 35J60}
\keywords{semi-linear, Harnack inequality, blowup solutions, mean curvature, scalar curvature, Yamabe type equations, energy estimate}

\author{Ying Guo}
\address{Department of Mathematics\\
       University of Florida\\
        358 Little Hall P.O.Box 118105\\
        Gainesville FL 32611-8105}
\email{yguo@ufl.edu}

\author{Lei Zhang}
\address{Department of Mathematics\\
        University of Florida\\
        358 Little Hall P.O.Box 118105\\
        Gainesville FL 32611-8105}
\email{leizhang@ufl.edu}

\date{\today}

%%%%%%%%%%%%%%%%%%%%%%%%%%%%%%%%%%%%%%%%%%%%%
\begin{abstract}
For a class of semi-linear elliptic equations with critical Sobolev exponents and boundary conditions, we prove point-wise estimates for blowup solutions and energy estimates. A special case of this class of equations is a locally defined prescribing scalar curvature and mean curvature type equation.
\end{abstract}
%%%%%%%%%%%%%%%%%%%%%%%%%%%%%%%%%%%%%%%%%%%%%

\maketitle

\section{Introduction}

In this article we consider
\begin{equation}\label{equ23}
  \left\{
   \begin{aligned}
   -\Delta u&=g(u), ~~~~~~~B^{+}_{3},  \\
   \frac{\partial{u}}{\partial{x_n}}&=h(u), ~~~~~~~\partial B_{3}^+ \cap \partial \mathbb R^n_+,
   \end{aligned}
   \right.
\end{equation}
where $u>0$ is a positive continuous solution, $B_{3}^+$ is the upper half ball centered at the origin with radius $3$, $g$ is a continuous function on $(0,\infty)$ and $h$ is locally H\"older continuous on $(0,\infty)$.

If $g(s)=s^{\frac{n+2}{n-2}}$ and $h(s)=cs^{\frac{n}{n-2}}$, the equation (\ref{equ23}) is a typical prescribing curvature equation. If we use $\delta$ to represent the Euclidean metric, then $u^{\frac{4}{n-2}}\delta$ is conformal to $\delta$. Equation (\ref{equ23}) in this special case means the scalar curvature under the new metric is $4(n-1)/(n-2)$ and the boundary mean curvature under the new metric is $-\frac{2}{n-2}c$.  Equation (\ref{equ23}) is very closely related to the well known Yamabe problem and the boundary Yamabe problem.   For $g$ and $h$ we assume
$$GH_0:\quad g \mbox{ is a continuous function on } (0,\infty), \quad h \mbox{ is H\"older continuous on } (0,\infty). $$
and
$$GH_1:\quad \left\{\begin{array}{ll}
\lim_{s\to\infty} g(s)s^{-\frac{n+2}{n-2}} \mbox{ is non-increasing},\quad \lim_{s\to \infty} g(s)s^{-\frac{n+2}{n-2}}\in (0,\infty).\\
s^{-\frac n{n-2}}h(s) \mbox{ is non-decreasing and } \lim_{s\to \infty} s^{-\frac{n}{n-2}}h(s)<\infty.
\end{array}
\right.
$$
Let
\begin{equation}\label{ha}
c_h:=\lim_{s\to \infty} s^{-\frac{n}{n-2}} h(s).
\end{equation}
Then if $c_h>0$ we assume
$$GH_2:\quad \sup_{0< s\le 1}g(s)s^{-1}<\infty, \quad \mbox{ and }\quad \sup_{0< s\le 1} s^{-1}|h(s)|<\infty. $$
If $c_h\le 0$ our assumption on $g,h$ is
$$GH_3: \quad \sup_{0< s\le 1}g(s)<\infty, \quad \mbox{ and } \quad \sup_{0< s\le 1} |h(s)|<\infty. $$
The main result of this article is concerned with the case $c_h>0$:
\begin{thm}\label{mainthm1} Let $u>0$ be a solution of (\ref{equ23}) where $g$ and $h$ satisfy $GH_0$ and $GH_1$. Suppose $c_h>0$ and $GH_2$ also holds, then
\begin{equation}\label{equ24}
\int _{B^{+}_{1}}|\nabla u|^2+u^{\frac{2n}{n-2}}\leq C,
\end{equation}
for some $C>0$ that depends only on g, h and n.
\end{thm}
Obviously if
\begin{equation}\label{spegh}
 g(s)=c_1s^{\frac{n+2}{n-2}},\,\, c_1>0\quad \mbox{ and }
  h(s)=c_h s^{\frac{n}{n-2}},\,\, c_h>0,
\end{equation}
 $g$ and $h$  satisfy the assumptions in Theorem \ref{mainthm1}. The energy estimate (\ref{equ24}) for this special case has been proved by Li-Zhang \cite{lei}. It is easy to see that the assumptions of $g$ and $h$ in Theorem \ref{mainthm1} include a much large class of functions. For example, for any non-increasing function $c_1(s)$ satisfying $\lim_{s\to \infty}c_1(s)>0$ and $\lim_{s\to 0+}c_1(s)s^{\frac{4}{n-2}}<\infty$,
$g(s)=c_1(s)s^{\frac{n+2}{n-2}}$ satisfies the assumptions of $g$. Similarly $h(s)=c_2(s)s^{\frac{n}{n-2}}$ for a nondecreasing function $c_2(s)$ with $\lim_{s\to \infty} c_2(s)=c_h$ and $\lim_{s\to 0+}|c_2(s)|s^{\frac{2}{n-2}}<\infty$, satisfies the requirement of $h$ in Theorem \ref{mainthm1}.

For the case $c_h\le 0$ we have
\begin{thm}\label{minorthm}
Let $u>0$ be a solution of (\ref{equ23}) where $g$ and $h$ satisfy $GH_0$ and $GH_1$. Suppose $c_h\le 0$ and $g,h$ satisfy $GH_3$,
then the energy estimate (\ref{equ24}) holds for
C depending only on g, h and n.
\end{thm}

If we allow $\lim_{s\to \infty} s^{-\frac{n+2}{n-2}}g(s)=0$, then the energy estimate (\ref{equ24}) may not hold. For example, let $g(s)=\frac 14(s+1)^{-3}$, then $g$ satisfies the assumption in Theorem \ref{minorthm} except that $\lim_{s\to \infty}s^{-\frac{n+2}{n-2}}g(s)=0$. Let
$u_j(x)=\sqrt{x_1+j}-1$, it is easy to verify that $u_j$ satisfies
$$\left\{\begin{array}{ll}
-\Delta u_j=g(u_j) \quad \mbox{ in }\quad B_3^+,\\
\displaystyle{\frac{\partial u_j}{\partial x_n}}=0, \quad \mbox{ on }\quad \partial B_3^+\cap \partial \mathbb R^n_+.
\end{array}
\right.
$$
Note that $h=0$ in this case. Then clearly (\ref{equ24}) does not hold for $u_j$.

The energy estimate (\ref{equ24}) is closely related to the following Harnack type inequality:
\begin{equation}\label{minmax}
(\min_{B_1^+}u )(\max_{B_2^+}u )\le C,
\end{equation}
which was proved by Li-Zhang \cite{lei} for the special case (\ref{spegh}). Li-Zhang \cite{lei} also proved the (\ref{equ24}) for (\ref{spegh}) using (\ref{minmax}) in their argument in a nontrivial way.

In the past two decades Harnack type inequalities similar to (\ref{minmax}) have played an important role in
blowup analysis for semilinear elliptic equations with critical Sobolev exponents. Pioneer works in this respect can be found in Schoen \cite{schoen1}, Schoen-Zhang \cite{schoen-zhang}, Chen-Lin
\cite{chen-lin-cpam1}, and further extensive results can be found in \cite{chen-lin-2,gui-lin,aobing,lei,li-zhang-4d,lin-pra,schoen-zhang,zhang2} and the references therein. Usually for a semi-linear equation without boundary condition, for example the conformal scalar curvature equation
$$\Delta u + K(x) u^{\frac{n+2}{n-2}}=0, \quad B_3, $$
a Harnack inequality of the type
$$\big (\max_{B_1} u \big )\big ( \min_{B_2} u \big ) \le C $$
immediately leads to the energy estimate
$$\int_{B_1} |\nabla u|^2+u^{\frac{2n}{n-2}}\le C $$
by the Green's representation theorem and integration by parts. However, when the boundary condition as in (\ref{equ23}) appears, using the Harnack inequality (\ref{minmax}) to derive (\ref{equ24}) is much more involved. In order to derive energy estimate (\ref{equ24}) and pointwise estimates for blow up solutions, Li and Zhang prove the following results in \cite{lei}:

\bigskip
\emph {Theorem A (Li-Zhang). Let $u>0$ be a solution of (\ref{equ23}) where $g$ and $h$ satisfy $GH_0$, $GH_1$ and $GH_3$. Then
$$(\max_{\overline{B_1^+}}u)(\min_{\overline{B_2^+}}u)\le C. $$ }

\medskip

Here we note that in Theorem $A$ no sign of $c_h$ is specified.
One would expect the energy estimate (\ref{equ24}) to follow directly from Li-Zhang's theorem. This is indeed the case if $c_h\le 0$. However for $c_h>0$ substantially more estimates are needed in order to establish a precise point-wise estimate for blowup solutions. As a matter of fact we need to assume $(GH_2)$ instead of $(GH_3)$ in order to obtain (\ref{equ24}).

The organization of this article is as follows. In section two we prove Theorem \ref{mainthm1}. The idea of the proof is as follows. First we use a selection process to locate regions in which the bubbling solutions look like global solutions. Then we consider the interaction of the bubbling regions. Using delicate blowup analysis and Pohozaev identity we prove that bubbling regions must be a positive distance apart. Then the energy estimate (\ref{equ24}) follows. Even thought the main idea we use to prove Theorem \ref{mainthm1} is similar to Li-Zhang's proof of the special case $g(s)=A s^{\frac{n+2}{n-2}}, h(s)=c_h s^{\frac{n}{n-2}}$, there are a lot of technical difficulties for the more general case. For example
Li-Zhang's proof relies heavily on the fact that the equation is invariant under scaling, thus they don't need any classification theorem in their
moving sphere argument. However in the more general case the equation is not scaling invariant any more and we have to use the classification theorem of Caffarelli-Gidas-Spruck or Li-Zhu.  In section three we prove Theorem \ref{minorthm} using Theorem A and integration by parts.

\section{Proof of Theorem \ref{mainthm1}}

The proof of Theorem \ref{mainthm1} is by way of contradiction. Suppose there is no energy bound, then there exists a sequence $u_k$ such that \begin{equation}\label{aug14e1}
\int_{B_1}|\nabla u_k|^2+u_k^{\frac{2n}{n-2}}\to \infty.
\end{equation}

We claim that $\max_{B_{3/2}^+}u_k\to \infty$. Indeed, if this is not the case, which means there is a uniform bound for $u_k$ on
$B_{3/2}^+$, we just take a cut-off function $\eta\in C^{\infty}$ such that $\eta\equiv 1$ on $B_1^+$ and $\eta\equiv 0$ on $B_{3/2}^+\setminus B_1^+$ and $|\nabla \eta|\le C$. Multiplying the equation (\ref{equ23}) by $u_k\eta^2$, using integration by parts and simple Cauchy's inequality we obtain a uniform bound of $\int_{B_1}|\nabla u_k|^2$, a contradiction to (\ref{aug14e1}).

\begin{Def}
Let $\{u_k\}$ be a sequence of solutions of (\ref{equ23}). Assume that $x_k\rightarrow \bar{x} \in \overline{B_2^+}$ and $\lim_{k\rightarrow \infty}u_k(x_k)= \infty$. If there exist $C>0$ (independent of k) and $\bar{r}>0$ (independent of k) such that $u_k(x){|x-\bar{x}|}^{\frac{n-2}{2}}\leq C$ for $|x-\bar{x}|\leq \bar{r}$, then we say that $\bar{x}$ is an isolated blow\text{-}up point of $\{u_k\}$.
\end{Def}

\begin{prop} \label{prop3.1}
Let $\{u_k\}$ be a sequence of solutions of (\ref{equ23}) and
$$
\max_{x \in \overline{B_1^+}}u_k(x){|x|}^{\frac{n-2}{2}}\rightarrow \infty ~~ as ~~ k\rightarrow \infty,
$$
then there exist a sequence of local maximum points $\bar x_k$ such that along a subsequence
(still denoted as $\{u_k\}$)
$$
v_k(y):={u_k(\bar x_k)}^{-1}u_k({u_k(\bar x_k)}^{-\frac{2}{n-2}}y+\bar x_k)
$$
either
converges uniformly over all compact subsets of $\mathbb R^n$ to $V$ that satisfies
\begin{equation}\label{global1}
\Delta V+ A V^{\frac{n+2}{n-2}}=0, \quad \mathbb R^n
\end{equation}
(where $A=\lim_{s\to \infty} g(s)s^{-\frac{n+2}{n-2}}$),
or converges to $V_1$ defined on $\{y\in \mathbb R^n;\quad y_n>-T\}$
( $T:=\lim_{k\to \infty} u_k(\bar x_k)^{\frac{2}{n-2}}\bar x_{kn}$ )
that satisfies
\begin{equation}\label{global2}
\left\{\begin{array}{ll}
\Delta V_1+ A V_1^{\frac{n+2}{n-2}}=0,\quad \mathbb R^n\cap \{y_n>-T\}, \\
\partial V_1/\partial y_n=c_h V_1^{\frac{n}{n-2}},\quad y_n=-T.
\end{array}
\right.
\end{equation}
where $c_h=\lim_{s\to \infty} s^{-\frac{n}{n-2}}h(s)$.
\end{prop}

\begin{rem}
All solutions of (\ref{global1}) are described by the classification theorem of Caffarelli-Gidas-Spruck \cite{caffarelli}. All solutions of
(\ref{global2}) are described by Li-Zhu \cite{liandzhu}.
\end{rem}

\noindent{\bf Proof of Proposition \ref{prop3.1}:}

Let $x_k\in B_1^+$ be a sequence such that
$$
u_k(x_k){|x_k|}^{\frac{n-2}{2}}\rightarrow \infty ~~ as ~~ k\rightarrow \infty.
$$
Let $d_k=|x_k|$ and $S_k(y)=u_k(y)(d_k-|y-x_k|)^{\frac{n-2}{2}}, \forall y \in B_1^+$.
Set
$$S_k(\widehat x_k)=\max_{\overline{{B(x_k, d_k) \cap \{t> -x_{kn}\}}}}S_k. $$ Then
\begin{equation}\label{13aug9e1}
S_k(\widehat x_k)\ge S_k(0)=u_k(x_k)|x_k|^{\frac{n-2}2}\rightarrow \infty.
\end{equation}

Let $\sigma_k=\frac{1}{2}(d_k-|x_k-\widehat x_k|)$, then clearly (\ref{13aug9e1}) can be written as
\begin{equation}\label{equ6}
u_k(\widehat x_k)2^{\frac{n-2}{2}}\sigma_k^{\frac{n-2}{2}}\geq u_k(x_k){d_k}^{\frac{n-2}{2}} \rightarrow \infty ~~ \mbox{as} ~~ k\rightarrow \infty.
\end{equation}

For all $x \in B_{\sigma_k}^+(\widehat x_k)$, since
$$
u_k(x) {(d_k-|x-x_k|)}^{\frac{n-2}{2}} \leq  u_k(\widehat x_k) {(d_k-|x_k-\widehat x_k|)}^{\frac{n-2}{2}},
$$
we have $$
u_k(x) \leq u_k(\widehat x_k) \left(\frac{d_k-|x_k-\widehat x_k|}{d_k-|x-x_k|} \right)^{\frac{n-2}{2}}.
$$
Using $|x-\hat x_k|\le \sigma_k$, and
$$d_k-|x-x_k|\ge d_k-|x_k-\hat x_k|-|x-\hat x_k|\ge \sigma_k, $$
we obtain
\begin{equation}\label{equ7}
u_k(x) \leq 2^{\frac{n-2}{2}}u_k(\widehat x_k), \quad for ~ all ~ x \in B_{\sigma_k}^+(\widehat x_k).
\end{equation}

Let $M_k=u_k(\widehat x_k)$ and
$$
v_k(y)={M_k}^{-1}u_k({M_k}^{-\frac{2}{n-2}}y+\widehat x_k), \quad M_k^{-\frac 2{n-2}}y+\hat x_k\in B_3^+.
$$
Direct computation shows
\begin{equation}\label{13aug9e3}
\Delta v_k(y)+{(M_kv_k(y))}^{-\frac{n+2}{n-2}}g(M_kv_k(y))\cdot {v_k(y)}^{\frac{n+2}{n-2}}=0
\end{equation}

By (\ref{equ7}) we have
\begin{equation}\label{13aug9e2}
0 \leq v_k(y) \leq 2^{\frac{n-2}{2}}\quad \forall y\in B(0,{M_k}^{\frac{2}{n-2}}\sigma_k)  \cap  \{y_n \geq -{M_k}^{\frac{2}{n-2}}{\widehat x _{kn}} \}
\end{equation}

We consider the following two cases.

\noindent{\bf Case one:
$\lim_{k \rightarrow \infty} {M_k}^{\frac{2}{n-2}}{\widehat x _{kn}} =\infty. $}

Since $M_k^{\frac 2{n-2}}\sigma_k$ and $M_k^{\frac 2{n-2}}\hat x_{nk}$ both tend to infinity,  (\ref{13aug9e3})
is defined on $|y|\le l_k$ for some $l_k\to \infty$. By (\ref{13aug9e2}) we assume that $v_k$ is bounded above in $B_{l_k}$. We claim that $v_k\to V$ uniformly over all compact subsets of $\mathbb R^n$ and $V$ satisfies (\ref{global1}) with $A=\lim_{s\to \infty} s^{-\frac{n+2}{n-2}}g(s)$.
Indeed, we claim that for any $R>1$,
\begin{equation}\label{localbd}
v_k(y)\ge C(R)>0,\quad |y|\le R.
\end{equation}
Once (\ref{localbd}) is established, we see clearly that $M_kV_k\to \infty$ over all $B_R$, thus
$$M_k^{-\frac{n+2}{n-2}}g(M_k v_k)=(M_k v_k)^{-\frac{n+2}{n-2}}g(M_k v_k)v_k^{\frac{n+2}{n-2}}\to AV^{\frac{n+2}{n-2}}$$
over all compact subsets of $\mathbb R^n$. Then it is easy to see that $V$ solves (\ref{global1}).

Therefore we only need to establish (\ref{localbd}) for fixed $R>1$.
Let
$$\Omega_{R,k}:=\{y\in B_R;\quad v_k(y)\le 3M_k^{-1}\quad \} $$
and
$$a_k(y)=M_k^{-\frac{n+2}{n-2}}g(M_k v_k)/v_k. $$
It follows from $(GH_1)$ that in $B_R\setminus \Omega_{R,k}$
$$a_k(y)\le g(3)v_k^{\frac{4}{n-2}}\le 4g(3). $$
For $y\in \Omega_{R,k}$ we use $(GH_2)$ to obtain
$$a_k(y)\le C M_k^{-\frac 4{n-2}},\quad y\in \Omega_{R,k}. $$
In either case $a_k(y)$ is a bounded function . From
$$\Delta v_k(y)+a_k(y)v_k(y)=0 \quad \mbox{ in } B_R $$
and standard Harnack inequality we have
$$1=v_k(0)\le \max_{B_{R/2}} v_k\le C(R) \min_{B_{R/2}}v_k. $$
Thus (\ref{localbd}) is established.
Consequently $V$, as the limit of $v_k$ indeed solves
(\ref{global1}).
By the classification theorem of Caffarelli-Gidas-Spruck \cite{caffarelli}
$$
V(y)=(\frac{n(n-2)}A)^{\frac{n-2}4}{\left (\frac{\mu}{1+{\mu}^2{|y|}^2} \right )}^{\frac{n-2}{2}}
$$
Obviously $V$ has a maximum point $\bar x \in \mathbb R^n$. Correspondingly there exists a sequence of local maximum points of $u_k$, denoted $\bar x_k$, that tends to $\bar x$ after scaling. Thus $v_k$ can be defined as in the statement of Proposition \ref{prop3.1}.

\medskip

{\bf Case two: $\lim_{k\to \infty} M_k^{\frac 2{n-2}} \hat x_{kn}<\infty$. }

In this case we let
$$ T=\lim_{k \rightarrow \infty} M_k^{\frac{2}{n-2}} \widehat x_{kn}. $$
It is easy to verify that $v_k$ satisfies
$$\left\{\begin{array}{ll}
\Delta v_k(y)+(M_kv_k(y))^{-\frac{n+2}{n-2}}g(M_k v_k(y)) v_k^{\frac{n+2}{n-2}}(y)=0, \quad \mbox{ in }\quad
M_k^{-\frac 2{n-2}}y+\hat x_k\in B_3^+, \\
\frac{\partial v_k}{\partial y_n}=(M_k v_k(y))^{-\frac{n}{n-2}} h(M_k v_k(y)) v_k(y)^{\frac 2{n-2}} v_k(y),
\quad \mbox{ on } y_n=-M_k^{\frac{2}{n-2}}\hat x_{kn}.
\end{array}
\right.
$$
We claim that for any $R>1$, there exists $C(R)>0$ such that
\begin{equation}\label{13aug15e1}
v_k(y)\ge C(R)\mbox{ in } B_R\cap\{y_n\ge -M_k^{\frac 2{n-2}} \hat x_{kn}\}.
\end{equation}

The proof of (\ref{13aug15e1}) is similar to the interior case. Let
$T_k=M_k^{\frac 2{n-2}}x_{kn}$ and $p_k=(0',-T_k)$. On $B(p_k, R)\cap \{y_n\ge -T_k\}$ we write the equation for $v_k$ as
\begin{equation}\label{fk1}
\left\{\begin{array}{ll}
\Delta v_k+a_k v_k=0 ,\quad \mbox{ in } B(p_k,R)\cap\{y_n> -T_k\}, \\
\\
\partial_n v_k+b_k v_k=0,\quad \mbox{ on } B(p_k,R)\cap \{y_n=-T_k\}.
\end{array}
\right.
\end{equation}
where it is easy to use $GH_2$ to prove that $|a_k|+|b_k|\le C$ for some $C$ independent of $k$ and $R$.
By a classical Harnack inequality with boundary terms (for example, see Lemma 6.2 of \cite{zhang1} or Han-Li \cite{hanli}), we have
$$1=v_k(0)\le \max_{B(p_k, R/2)\cap \{y_n\ge -T_k\}}v_k\le C(R)\min_{B(p_k,R/2)\cap \{y_n\ge -T_k\}} v_k. $$
Therefore $v_k$ is bounded below by positive constants over all compact subsets.
 Thus the limit function $V_1$ solves (\ref{global2}).  By
Li-Zhu's classification theorem \cite{liandzhu},
$$
V_1(y)=(\frac{n(n-2)}A)^{\frac{n-2}4}{\left (\frac{\lambda}{1+\lambda^2({|y'-\bar y|}^2+{|y_n-\bar y_n|}^2)} \right )}^{\frac{n-2}{2}},
$$
where $\bar y_n=c_h\sqrt{(n-2)n/A}/((n-2)\lambda)$, $\bar y\in \mathbb R^{n-1}$, $\lambda$ is determined by $V_1(0)=1$. Thus the local maximum of $V_1$ can be used to defined $v_k$ as in the statement of the proposition.
Proposition \ref{prop3.1} is established. $\Box $

\medskip

Proposition \ref{prop3.1} determines the first point in the blowup set $\Sigma_k$. The other points in $\Sigma_k$ can be determined as follows: Consider the maximum of
$$S_k(x)=u_k(x)^{\frac{2}{n-2}}dist(x,\Sigma_k). $$
If $S_k(x)$ is uniformly bounded we stop. Otherwise the same selection process we get another blowup profile by either the classification theorem of Caffarelli-Gidas-Spruck or Li-Zhu.  Eventually we have $\{q_k^i\}\in \Sigma_k$ ($i=1,2,..,$) that satisfy
$$
  \left\{
   \begin{aligned}
   &\overline{B_{r_i^k}^+(q_i^k)} \cap \overline{B_{r_j^k}^+(q_j^k)}=\emptyset, & for ~ i \neq j,  \\
   &{|q_i^k-q_j^k|}^{\frac{n-2}{2}}u_k(q_j^k)\to \infty, & for ~ j > i,\\
   &u_k(x)^{\frac{2}{n-2}} dist (x,\Sigma_k)\le C.
   \end{aligned}
   \right.
$$
and $r_i^k$ are chosen so that in $B_{r_i^k}^+(q_i^k)$, the profile of $u_k$ is either like an entire standard bubble described in (\ref{global1})
or a part of the bubble described in (\ref{global2}).

\medskip

Take any $q_k\in \Sigma_k$, let $\sigma_k=dist(q_k,\Sigma_k\setminus \{q_k\})$ and we let
$$\tilde u_k(y)=\sigma_k^{\frac{n-2}2}u_k(q_k+\sigma_k y),\quad \mbox{ in } \Omega_k $$
where $\Omega_k:=\{y;\quad q_k+\sigma_k y\in B_3^+ \}$. By the selection process we have
\begin{equation}\label{aug15e10}
\tilde u_k(y)\le C|y|^{-\frac{n-2}2}, \quad |y|\le 3/4, y\in \Omega_k
\end{equation}
 and
 \begin{equation}\label{aug15e11}
\tilde u_k(0)\to \infty.
\end{equation}

We further prove in the following proposition that $\tilde u_k$ decays like a harmonic function:

\begin{prop}\label{prop3.2}
\begin{equation}\label{isosimple1}
\tilde u_k(0)\tilde u_k(y){|y|}^{n-2} \leq C, \quad \mbox{ for } y\in B_{2/3}\cap \Omega_k.
\end{equation}
\end{prop}

\begin{rem} The meaning of Proposition \ref{prop3.2} is each isolated blowup point is also isolated simple.
\end{rem}

\begin{proof}

Direct computation shows that $\tilde u_k$ satisfies
\begin{equation}\label{tildeuk}
\left\{\begin{array}{ll}
\Delta \tilde u_k(y)+\sigma_k^{\frac{n+2}2}g(\sigma_k^{-\frac{n-2}2}\tilde u_k)=0, \quad \mbox{ in }\Omega_k, \\
\\
\partial_n \tilde u_k(y)=\sigma_k^{\frac n2}h(\sigma_k^{-\frac{n-2}2}\tilde u_k),\quad \mbox{ on }\quad \partial\Omega_k
\cap \{y_n=-\sigma_k^{-1}q_{kn}\},\\
\end{array}
\right.
\end{equation}

Let $\tilde M_k=\tilde u_k(0)$. By (\ref{aug15e11}) $\tilde M_k\to \infty$. Set
$$ v_k(z)=\tilde M_k^{-1} \tilde u_k(\tilde M_k^{-\frac 2{n-2}}z), \quad \mbox{ for } z\in \tilde \Omega_k, $$
where
$$\tilde \Omega_k:=\{z; \quad |z|\le \tilde M_k^{\frac 2{n-2}},\quad \tilde M_k^{-\frac 2{n-2}}z\in \Omega_k\} $$
Note that $v_k$ is defined on a bigger set, but for the proof of Proposition \ref{prop3.2} we only need to consider the part in
$\tilde \Omega_k$.

Direct computation gives
\begin{equation}\label{c2vk}
\left\{\begin{array}{ll}
\Delta v_k(z)+l_k^{-\frac{n+2}{n-2}}g(l_k v_k)=0,\quad z\in \tilde \Omega_k, \\
\\
\frac{\partial v_k}{\partial z_n}=l_k^{-\frac n{n-2}}h(l_k v_k),\quad \{z_n=-T_k\}\cap \partial \tilde \Omega_k.
\end{array}
\right.
\end{equation}
where $l_k=\sigma_k^{-\frac{n-2}2}\tilde M_k$ and $T_k=l_k^{\frac 2{n-2}}q_{kn}$.
We consider two cases.

\noindent{\bf Case one: $T_k\to \infty$. }

By the same argument as in the proof of Proposition \ref{prop3.1}, we know $v_k\to V$ in $C^{1,\alpha}_{loc}(\mathbb R^n)$ where $V$ solves (\ref{global1}).
Thus there exist $R_k\to \infty$ such that
$$\|v_k-V\|_{C^{1,\alpha}(B_{R_k})}\le CR_k^{-1}. $$
Clearly (\ref{isosimple1}) holds for $|z|\le \tilde M_k^{-\frac{2}{n-2}}R_k$, we just need to prove (\ref{isosimple1}) for $|z|>\tilde M_k^{-\frac{2}{n-2}}R_k$. Since $V(0)=1$ is the maximum point of $V$,
$$V(z)=(1+\frac{A}{n(n-2)}|z|^2)^{-\frac{n-2}2}, \quad z\in \mathbb R^n. $$

\begin{lem}\label{har1} There exists $k_0>1$ such that for all $k\ge k_0$ and
$r\in (R_k, \tilde M_k^{\frac 2{n-2}})$,
\begin{equation}\label{aug16e1}
\min_{\partial B_r\cap \tilde \Omega_k} v_k\le 2(\frac{n(n-2)}A)^{\frac{n-2}2}r^{2-n}.
\end{equation}
\end{lem}

\noindent{\bf Proof of Lemma \ref{har1}:}

Suppose (\ref{aug16e1}) does not hold, then there exist $r_k$ such that
\begin{equation}\label{aug16e2}
v_k(z)\ge 2(\frac{n(n-2)}A)^{\frac{n-2}2}r_k^{2-n}, \quad |z|=r_k, z\in \tilde \Omega_k.
\end{equation}
Clearly $r_k\ge R_k$.

Let
$$v_k^{\lambda}(z)=(\frac{\lambda}{|z|})^{n-2}v_k(z^{\lambda}), \quad z^{\lambda}=\frac{\lambda^2 z}{|z|^2}. $$
The equation of $v_k^{\lambda}$, by direct computation, is
\begin{equation}\label{vlam}
\Delta v_k^{\lambda}(z)+(\frac{\lambda}{|z|})^{n+2}l_k^{-\frac{n+2}{n-2}}g(l_k(\frac{|z|}{\lambda})^{n-2}v_k^{\lambda}(z))=0,
\quad \mbox{ in }\Sigma_{\lambda}
\end{equation}
where
$$\Sigma_{\lambda}:=\{z\in \tilde \Omega_k;\quad |\lambda|<|z|<r_k \,\, \}.$$
Clearly $v_k^{\lambda}\to V^{\lambda}$ in $C^{1,\alpha}_{loc}(\mathbb R^n)$ for fixed $\lambda>0$.
By direct computation
\begin{align*}
V(z)>V^{\lambda}(z), \mbox{ for } \lambda\in (0,(\frac{n(n-2)}A)^{1/2}), \quad |z|>\lambda \\
V(z)<V^{\lambda}(z), \mbox{ for } \lambda>(\frac{n(n-2)}A)^{1/2}, \quad |z|>\lambda.
\end{align*}
We shall apply the method of moving spheres for $\lambda\in (\frac 12(\frac{n(n-2)}A)^{1/2}, 2(\frac{n(n-2)}A)^{1/2})$. First we prove that for $\lambda_0=\frac 12(\frac{n(n-2)}A)^{1/2}$,
\begin{equation}\label{aug16e3}
v_k(z)>v_k^{\lambda_0}(z),\quad z\in \Sigma_{\lambda}.
\end{equation}
To prove (\ref{aug16e3}) we first observe that $v_k>v_k^{\lambda_0}$ in $B_R\setminus B_{\lambda}$ for any fixed $R$ large. Indeed, $v_k=v_k^{\lambda_0}$ on $\partial B_{\lambda_0}$. On $\partial B_{\lambda_0}$ we have $\partial_{\nu}V>\partial_{\nu}V^{\lambda_0}$. Thus the $C^{1,\alpha}$ convergence of $v_k$ to $V$ gives that $v_k>v_k^{\lambda_0}$ near $\partial B_{\lambda_0}$. Then by the uniform convergence we further know that (\ref{aug16e3}) holds on $B_R\setminus B_{\lambda_0}$. On $\partial B_R$, we have
\begin{equation}\label{sept3e2}
v_k(z)\ge ((\frac{n(n-2)}A)^{\frac{n-2}2}-\epsilon)|z|^{2-n}, \quad |z|=R
\end{equation}
and
\begin{equation}\label{sept3e1}
v_k^{\lambda_0}(z)\le ((\frac{n(n-2)}A)^{\frac{n-2}2}-2\epsilon)|z|^{2-n}, \quad |z|\ge R
\end{equation}
for some $\epsilon>0$ independent of $k$.
Next we shall use maximum principle to prove that
\begin{equation}\label{greaterhar}
v_k(z)>((\frac{n(n-2)}A)^{\frac{n-2}2}-2\epsilon)|z|^{2-n}> v_k^{\lambda_0}(z),\quad z\in \Sigma_{\lambda_0}\setminus B_R.
\end{equation}
The proof of (\ref{greaterhar}) is by contradiction. We shall compare $v_k$ and
$$f_k:=((\frac{n(n-2)}A)^{\frac{n-2}2}-2\epsilon)|z|^{2-n}. $$ Clearly $v_k-f_k$ is super harmonic in $\Sigma_{\lambda_0}-B_R$ and,
by (\ref{sept3e2}),(\ref{sept3e1}) and (\ref{aug16e2}), $v_k-f_k>0$ on $\partial B_R$ and $\partial \Sigma_{\lambda_0}\cap (\mathbb R^n_+\setminus B_R)$. If there exists $z_0\in \partial\Sigma_{\lambda_0}\cap \{z_n=-T_k\}$ and

$$0>v_k(z_0)-f_k(z_0)=\min_{\Sigma_{\lambda_0}\setminus B_R}v_k-f_k $$
we would have
\begin{equation}\label{aug16e11}
0<\partial_n(v_k-f_k)(z_0)=l_k^{-\frac{n}{n-2}}h(l_k v_k(z_0))-\partial_n f_k(z_0).
\end{equation}
Easy to verify $\partial_n f_k(z_0)>N_k f_k(z_0)^{\frac{n}{n-2}}$ for some $N_k\to \infty$. However by $GH_1$
$$l_k^{-\frac{n}{n-2}}h(l_k v_k(z_0))\le C v_k(z_0)^{\frac n{n-2}}. $$
Easy to see it is impossible to have $v_k(z_0)<f_k(z_0)$ and (\ref{aug16e11}). (\ref{greaterhar}) is established.

Before we employ the method of moving spheres we set
$$O_{\lambda}:=\{z\in \Sigma_{\lambda};\quad v_k(z)<\min\bigg ( (\frac{|z|}{\lambda})^{n-2},2\bigg) v_k^{\lambda}(z)\quad \}. $$
This is the region where maximum principle needs to be applied.

By $GH_1$ we have
\begin{align*}
&(\frac{\lambda}{|z|})^{n+2} l_k^{-\frac{n+2}{n-2}} g(l_k (\frac{|z|}{\lambda})^{n-2}v_k^{\lambda})\\
=& \bigg ( (\frac{|z|}{\lambda})^{n-2} l_k v_k^{\lambda} )^{-\frac{n+2}{n-2}}g((\frac{|z|}{\lambda})^{n-2} l_k v_k^{\lambda})
(v_k^{\lambda})^{\frac{n+2}{n-2}}\\
\le & (l_k v_k)^{-\frac{n+2}{n-2}}g(l_k v_k)(v_k^{\lambda})^{\frac{n+2}{n-2}},\quad \mbox{ in } O_{\lambda}.
\end{align*}
Therefore in $O_{\lambda}$ we have
\begin{equation}\label{aug19e1}
\Delta v_k^{\lambda}+(l_k v_k)^{-\frac{n+2}{n-2}}g(l_k v_k) (v_k^{\lambda})^{\frac{n+2}{n-2}}\ge 0,\quad \mbox{ in } O_{\lambda}.
\end{equation}
The equation for $v_k$ can certainly be written as
\begin{equation}\label{aug19e2}
\Delta v_k +(l_k v_k)^{-\frac{n+2}{n-2}} g(l_k v_k) v_k^{\frac{n+2}{n-2}}=0.
\end{equation}
Let $w_{\lambda,k}=v_k-v_k^{\lambda}$, we have, from (\ref{aug19e1}) and (\ref{aug19e2})
\begin{equation}\label{wlam}
\Delta w_{\lambda,k}+n(n-2)(l_k v_k)^{-\frac{n+2}{n-2}}g(l_k v_k) \xi_k^{\frac{4}{n-2}} w_{\lambda,k}\le 0, \quad \mbox{ in } O_{\lambda},
\end{equation}
where $\xi_k$ is obtained from the mean value theorem.

Now we apply the method of moving spheres to $w_{\lambda,k}$. Let
$$\bar \lambda_k=\inf \{\lambda\in [(\frac{n(n-2)}A)^{\frac 12}-\epsilon_0, (\frac{n(n-2)}A)^{\frac 12}+\epsilon_0];\quad v_k>v_k^{\mu} \mbox{ in }\Sigma_{\mu}, \forall \mu>\lambda\,\, \} $$
where $\epsilon_0>0$ is chosen to be independent of $k$ and
\begin{equation}\label{sept16e1}
v_k(z)>v^{\lambda}(z),\quad |z|=r_k,\quad z\in \tilde \Omega_k,\quad \forall \lambda\in [\lambda_0,\lambda_1].
\end{equation}
From (\ref{aug16e2}) we see that $\epsilon_0$ can be chosen easily.
By (\ref{aug16e3}), $\bar \lambda_k>(\frac{n(n-2)}A)^{\frac 12}-\epsilon_0$. We claim that $\bar \lambda_k=(\frac{n(n-2)}A)^{\frac 12}+\epsilon_0$. Suppose this is not the case we have $\bar\lambda_k<(\frac{n(n-2)}A)^{\frac 12}+\epsilon_0$. By continuity $w_{\bar\lambda_k,k}\ge 0$ and by (\ref{aug16e2}) $w_{\bar \lambda_k,k}>0$ on the outside boundary: $\partial \Sigma_{\bar \lambda_k}\setminus
(\partial B_{\bar \lambda_k}\cup \{z_n=-T_k\})$.
By (\ref{wlam}), if $\min_{\bar \Sigma_{\bar \lambda_k}}w_{\bar \lambda_k,k}=0$, the minimum will have to appear on $\partial \Sigma_{\bar \lambda_k}$. From (\ref{sept16e1}) we see that the minimum does not appear on $\partial \Sigma_{\bar \lambda_k}\setminus (\partial B_{\bar \lambda_k}
\cup \{z_n=-T_k\})$. If there exists $x_0\in \partial \Sigma_{\bar \lambda_k}\cap \{z_n=-T_k\}$, we have
$$0<\partial_{z_n}w_{\bar \lambda_k,k}(x_0) $$
Note that we have strict inequality because of Hopf Lemma. On the other hand, using $T_k\to \infty$ and elementary estimates we have
$$\frac{\partial v_k^{\bar \lambda_k}}{\partial z_n}>N_k (v_k^{\bar \lambda_k})^{\frac n{n-2}}, \quad \mbox{ in } O_{\bar \lambda_k}\cap
\{z_n=-T_k\},\quad
\mbox{ for some } N_k\to \infty. $$
For $v_k$, $GH_1$ implies
$$\partial_{z_n}v_k\le c_h v_k^{\frac{n}{n-2}},\quad \mbox{ in } O_{\bar \lambda_k}\cap \{z_n=-T_k\} $$
where $c_h=\lim_{s\to \infty} s^{-\frac n{n-2}} h(s)$.
Then it is easy to see that $w_{\bar \lambda_k}>0$ on $\{z_n=-T_k\}$. Then Hopf Lemma and the continuity lead to a contradiction of the definition of $\bar \lambda_k$.  Thus we have proved $\bar \lambda_k=(\frac{n(n-2)}A)^{\frac 12}+\epsilon_0$. However $V<V^{\lambda}$ for $|y|>\lambda$ if $\lambda>(\frac{n(n-2)}A)^{\frac 12}$. So it is impossible to have
$\lim_{k\to \infty} \bar \lambda_k>(\frac{n(n-2)}A)^{\frac 12}$. This contradiction proves (\ref{aug16e1}) under {\bf Case one}.
Lemma \ref{har1} is established.

\medskip

From Lemma \ref{har1} we further prove the spherical Harnack inequality for $v_k$. For fixed $k$, consider $2R_k\le r\le \frac 12 \tilde M_k^{\frac 2{n-2}}$ and let
$$\tilde v_k(z)=r^{\frac{n-2}2}v_k(rz). $$
By (\ref{aug15e10}), $\tilde v_k(z)\le C$.
Direct computation yields
$$\left\{\begin{array}{ll}
\Delta \tilde v_k(z)+r^{\frac{n+2}2}l_k^{-\frac{n+2}{n-2}}g(l_k r^{-\frac{n-2}2} \tilde v_k)=0, \quad \frac 12<|z|<2, rz\in \tilde \Omega_k,\\
\\
\partial_n \tilde v_k=r^{\frac n2}l_k^{-\frac n{n-2}}h(r^{-\frac{n-2}2}l_k\tilde v_k),\quad \partial' \tilde \Omega_k.
\end{array}
\right.
$$
Let
\begin{align*}
a_k=r^{\frac{n+2}2}l_k^{-\frac{n+2}{n-2}}g(l_k r^{-\frac{n-2}2} \tilde v_k)/\tilde v_k \\
b_k=r^{\frac n2}l_k^{-\frac n{n-2}}h(r^{-\frac{n-2}2}l_k\tilde v_k)/\tilde v_k
\end{align*}
By the definition of of $l_k$ and $r$ we see that $r=o(1)l_k^{\frac{2}{n-2}}$.
Using the assumptions of $g$ we have
$$a_k(z)\le \left\{\begin{array}{ll}
g(1)\tilde v_k^{\frac{4}{n-2}}\le C,\quad  \mbox{ if } l_kr^{-\frac{n-2}2}\tilde v_k(z)\ge 1,\\
\\
Cr^2l_k^{-\frac{4}{n-2}}=o(1),\quad \mbox{if } l_kr^{-\frac{n-2}2}\tilde v_k(z)\le 1,
\end{array}
\right.
$$
and
$$|b_k(z)|\le \left\{\begin{array}{ll}
c_h \tilde v_k^{\frac{2}{n-2}}\le C,\quad \mbox{ if } l_kr^{-\frac{n-2}2}\tilde v_k(z)\ge 1 \\
\\
Crl_k^{-\frac{2}{n-2}}=o(1),\quad \mbox{if } l_kr^{-\frac{n-2}2}\tilde v_k(z)\le 1,
\end{array}
\right.
$$

Hence
$a_k$ and $b_k$ are both bounded functions.

Consequently the equation for $\tilde v_k$ can be written as
$$\left\{\begin{array}{ll}
\Delta \tilde v_k(z)+a_k \tilde v_k=0, \quad \frac 12<|z|<2, rz\in \tilde \Omega_k,\\
\\
\partial_n \tilde v_k=b_k \tilde v_k,\quad \partial \tilde \Omega_k\cap\{z_n=-T_k/r\}.
\end{array}
\right.
$$
Clearly we apply classical Harnack inequality for two cases: either $T_k/r>1$ or $T_k/r\le 1$. In the first case we have
$$\max_{|z|=3/4} \tilde v_k(z)\le C \min_{|z|=3/4}\tilde v_k. $$
In the second case we have
 $$\max_{|z|=1,z_n\ge -T_k/r} \tilde v_k(z)\le C \min_{|z|=1,z_n\ge -T_k/r}\tilde v_k. $$
 Clearly (\ref{isosimple1}) is implied. Proposition \ref{prop3.2} is established for {\bf Case one}.

\bigskip

{\bf Case two: $\lim_{k\to \infty} T_k=T$ }

\medskip

Recall that $v_k$ satisfies (\ref{c2vk}).

By the same argument as in the proof of Proposition \ref{prop3.1}, we know $v_k\to V$ in $C^{1,\alpha}_{loc}(\mathbb R^n\cap \{y_n\ge -T\})$ where $V$ solves (\ref{global2}).
Thus there exist $R_k\to \infty$ such that
$$\|v_k-V\|_{C^{1,\alpha}(B_{R_k}^{-T})}\le CR_k^{-1}. $$
Clearly (\ref{isosimple1}) holds for $|y|\le \tilde M_k^{-\frac{2}{n-2}}R_k\cap\{y_n\ge -T_k\}$, we just need to prove (\ref{isosimple1}) for $\{|y|>\tilde M_k^{-\frac{2}{n-2}}R_k\}\cap \{y_n\ge -T_k\}$. Since $V(0)=1$ is the maximum point of $V$,
$$V(z)=(1+\frac{A}{n(n-2)}|z|^2)^{-\frac{n-2}2}, \quad z\in \mathbb R^n. $$

\begin{lem}\label{har2} There exists $k_0>1$ such that for all $k\ge k_0$ and
$r\in (R_k, \tilde M_k^{\frac 2{n-2}})$,
\begin{equation}\label{aug16e1a}
\min_{\partial B_r\cap \tilde \Omega_k} v_k\le 2(\frac{n(n-2)}A)^{\frac{n-2}2}r^{2-n}.
\end{equation}
\end{lem}

\noindent{\bf Proof of Lemma \ref{har2}:}

Just like the interior case suppose there exist $r_k\ge R_k$ such that
\begin{equation}\label{har5}
\min_{\partial B_{r_k}\cap \tilde \Omega_k} v_k> 2(\frac{n(n-2)}A)^{\frac{n-2}2}r_k^{2-n}.
\end{equation}
Let
$$\tilde v_k(z)=v_k(z-T_ke_n), \quad \tilde v_k^{\lambda}(z)=(\frac{\lambda}{|z|})^{n-2}\tilde v_k(\frac{\lambda^2z}{|z|^2}) $$
and
$$D_k:=\{z; \quad \tilde M_k^{-\frac 2{n-2}}(z-T_ke_n)\in \Omega_k\cap B_{r_k} \} $$
be the domain of $\tilde v_k$. Then $D_k\subset \mathbb R^n_+$. Set
$$\Sigma_{\lambda}:=\{z\in D_k;\quad |z|>\lambda \}. $$
Let $\tilde V$ be the limit of $\tilde v_k$ in $C^2_{loc}(\mathbb R^n_+)$:
 $$\tilde V(z)=(1+\frac{A}{n(n-2)}|z-Te_n|^2)^{-\frac{n-2}2},$$
then there exists $\lambda_0<\lambda_1$ depending only on $n,A,T$ such that
$$\tilde V>\tilde V^{\lambda_0}\quad \mbox{ in }\mathbb R^n_+\setminus B_{\lambda_0} $$
and
$$\tilde V<\tilde V^{\lambda_1}\quad \mbox{ in }\mathbb R^n_+\setminus B_{\lambda_1}. $$

We shall employ the method of moving spheres to compare
$\tilde v_k$ and $\tilde v_k^{\lambda}$ on $\Sigma_{\lambda}$ for $\lambda\in [\lambda_0,\lambda_1]$.

First we use the uniform convergence of $\tilde v_k$ to $\tilde V$ to assert that, for any fixed $R>1$
\begin{equation}\label{case2e1}
\tilde v_k(y)>\tilde v_k^{\lambda_0}(y),\quad y\in \Sigma_{\lambda_0}\cap B_R .
\end{equation}

For $R$ large we have (let $a_1=(\frac{n(n-2)}A)^{\frac{n-2}2}$)
$$\tilde v_k(y)\ge (a_1-\epsilon/5)|y|^{2-n} \quad \mbox{ on } \partial B_R\cap \mathbb R^n_+ $$
and
$$\tilde v_k^{\lambda_0}(y)\le (a_1-2\epsilon/5)|y|^{2-n},\quad |y|>\lambda_0. $$
To prove $\tilde v_k>\tilde v_k^{\lambda_0}$ in $\Sigma_{\lambda_0}\setminus B_R$ we compare $\tilde v_k$ with
$$w= (a_1-3\epsilon/10) |y- A_1 e_n |^{2-n} $$
where $A_1=1/(n-2) c_h a_1^{\frac{2}{n-2}}$.
 For $R$ chosen sufficiently large we have
 $$w\ge \tilde v_k^{\lambda_0} \mbox{ in } \Sigma_{\lambda_0}\setminus B_R $$
 and
 $$\tilde v_k>w \quad \mbox{ on }\partial B_R\cap \Sigma_{\lambda_0}. $$

 To compare $\tilde v_k$ and $w$ over $\Sigma_{\lambda_0}\setminus B_R$, it is easy to see that $\tilde v_k>w$ on $\partial B_R\cap \Sigma_{\lambda_0}$ and $\partial \Sigma_{\lambda_0}\setminus (B_R\cup \{z_n>0\})$. Since $\tilde v_k-w$ is super-harmonic, the only thing we need to prove is on $\partial \mathbb R^n_+\setminus B_{\lambda_0}$
\begin{equation}\label{case2bry}
\partial_n (\tilde v_k-w)< \xi_k (\tilde v_k-w), \quad \mbox{ on } z_n=0 .
\end{equation}
for some positive function $\xi_k$.  Then standard maximum principle can be used to conclude that $\tilde v_k>w_k$ on $\Sigma_{\lambda_0}\setminus B_R$.

To obtain (\ref{case2bry}) first for $\tilde v_k$ we use $GH_2$ to have
$$\partial_n \tilde v_k\le c_h \tilde v_k^{\frac{n}{n-2}},\quad z_n=0. $$
On the other hand by the choice of $A_1$ we verify easily that
$$\partial_n w> c_h w^{\frac{n}{n-2}},\quad z_n=0. $$
Thus (\ref{case2bry}) holds from mean value theorem. We have proved that the moving sphere process can start at $\lambda=\lambda_0$:
$$\tilde v_k>\tilde v_k^{\lambda_0} \quad \mbox{ in }\quad \Sigma_{\lambda_0}. $$
Let $\bar \lambda$ be the critical moving sphere position:
$$\bar \lambda:=\min\{\lambda\in [\lambda_0,\lambda_1];\quad \tilde v_k^{\mu}>\tilde v_k^{\mu} \mbox{ in }\Sigma_{\mu}, \forall \mu>\lambda. \}. $$
As in {\bf Case one} we shall prove that $\bar \lambda=\lambda_1$, thus getting a contradiction from $\tilde V<\tilde V^{\lambda_1}$ for $|z|>\lambda_1$. For this purpose we let
$$w_{\lambda,k}=\tilde v_k-\tilde v_k^{\lambda}.  $$
To derive the equation for $w_{\lambda,k}$ we first recall from (\ref{c2vk}) and the definition of $\tilde v_k$ that
\begin{equation}\label{tvk}
\left\{\begin{array}{ll}
\Delta \tilde v_k(z)+l_k^{-\frac{n+2}{n-2}}g(l_k \tilde v_k)=0,\quad z\in \tilde \Omega_k, \\
\\
\frac{\partial \tilde v_k}{\partial z_n}=l_k^{-\frac n{n-2}}h(l_k \tilde v_k),\quad \{z_n=0\}\cap \partial \tilde \Omega_k.
\end{array}
\right.
\end{equation}
where $l_k=\sigma_k^{-\frac{n-2}2}\tilde M_k$. Correspondingly $\tilde v_k^{\lambda}$ satisfies
\begin{equation}\label{tvl}
\left\{\begin{array}{ll}
\Delta \tilde v_k^{\lambda}+(\frac{\lambda}{|z|})^{n+2}l_k^{-\frac{n+2}{n-2}}g(l_k(\frac{|z|}{\lambda})^{n-2}\tilde v_k^{\lambda}(z))=0,
\mbox{ in } \tilde \Sigma_{\lambda},\\
\\
\frac{\partial \tilde v_k^{\lambda}}{\partial z_n}=(\frac{\lambda}{|z|})^n l_k^{-\frac{n}{n-2}}h(l_k(\frac{|z|}{\lambda})^{n-2}\tilde v_k^{\lambda}(z))
\quad \mbox{ on } \partial \Sigma_{\lambda}\cap \{z_n=0\}.
\end{array}
\right.
\end{equation}
Let $O_{\lambda}$ be defined as before. Then in $O_{\lambda}$ we have, by $GH_1$,
\begin{align*}
(\frac{\lambda}{|z|})^{n+2}l_k^{-\frac{n+2}{n-2}}g(l_k(\frac{|z|}{\lambda})^{n-2}\tilde v_k^{\lambda}(z))\le
(v_kl_k)^{-\frac{n+2}{n-2}}g(l_k v_k)(v_k^{\lambda})^{\frac{n+2}{n-2}}, \mbox{ in } O_{\lambda} \\
(\frac{\lambda}{|z|})^n l_k^{-\frac{n}{n-2}}h(l_k(\frac{|z|}{\lambda})^{n-2}\tilde v_k^{\lambda}(z))\ge
(l_k v_k)^{-\frac n{n-2}} h(l_k v_k )(v_k^{\lambda})^{\frac n{n-2}},\mbox{ on } \partial O_{\lambda}\cap \{z_n=0\}.
\end{align*}
The inequalities above yield
\begin{align*}
\Delta w_{\lambda,k}+\xi_{1,k} w_{\lambda,k}\le 0,\quad \mbox{ in } O_{\lambda}\\
\partial_n w_{\lambda,k}\le \xi_{2,k} w_{\lambda,k},\quad \mbox{ on } \partial O_{\lambda}\cap \{z_n =0\}.
\end{align*}
where $\xi_{1,k}>0$ and $\xi_{2,k}$ are continuous functions obtained from mean value theorem. It is easy to see that the moving sphere argument can be employed to prove that $\bar \lambda=\lambda_1$, which leads to a contradiction from the limiting function $\tilde V$.
Thus
Lemma \ref{har2} is established. $\Box$

\medskip

Lemma \ref{har2} guarantees that on each radius $R_k\le r\le \frac 12 \tilde M_k$ the minimum of $v_k$ is always comparable to
$|z|^{2-n}$. Re-scaling $v_k$ as $r^{\frac{n-2}2}v_k(rz)$ we see the spherical Harnack inequality holds by the $GH_2$ and $GH_3$. Thus
Proposition \ref{prop3.2} is established in {\bf Case Two} as well.

\end{proof}

\begin{lem} \label{lem3.2}
Let $\{u_k\}$ be a sequence of solutions of (\ref{equ23}) and $q_k \rightarrow q \in \overline{B_1^+}$ be a sequence of points in $\Sigma_k$. Then there exist $C>0$, $r_2>0$ independent of $k$ such that
$$
u_k(q_k)u_k(x) \geq C{|x-q_k|}^{2-n} \mbox{ in } |x-q_k|\le r_2, \quad x\in B_3^+.
$$
\end{lem}
\begin{proof}

Without loss of generality we assume that $q_k$s are local maximum points of $u_k$. We consider two cases

\noindent{\bf Case one: $u_k(q_k)^{\frac 2{n-2}}q_{kn}\to \infty$. }

Let $M_k=u_k(q_k)$ and
\begin{equation}\label{vkd1}
v_k(y)={M_k}^{-1}u_k({M_k}^{-\frac{2}{n-2}}y+q_k), \quad y\in \Omega_k:=\{y; \,\, M_k^{-\frac{2}{n-2}}y+q_k\in B_3^+\}.
\end{equation}
Then in this case $v_k$ converges uniformly to
$$V(y)=(1+\frac{A}{n(n-2)}|y|)^{-\frac{n-2}2} $$
over all compact subsets of $\mathbb R^n$.
For $\epsilon>0$ small we let
$$\phi=(\frac{n(n-2)}A-\epsilon)^{\frac{n-2}2} (|y|^{2-n}-M_k^{-2}) \quad |y|\le M_k^{\frac{2}{n-2}} $$
on $|y|\ge R$ where $R>1$ is chosen so that $v_k>\phi$ on $\partial B_R$. By direct computation we have
$$\frac{\partial \phi}{\partial y_n}> N_k \phi^{\frac{n}{n-2}}, \quad \mbox{ on } y_n=-q_{kn}M_k $$
for some $N_k\to \infty$. It is easy to see that $v_k\ge \phi$ on $\partial \Omega_k\setminus \{y_n=-M_k^{\frac 2{n-2}}q_{kn}\}$. On
$\{y_n=-M_k^{\frac 2{n-2}}q_{kn}\}$ we have
$$\partial_{y_n}(v_k-\phi)\le c_h(v_k-\phi). $$
 Thus standard maximum principle implies $v_k\ge \phi$ on $\Omega_k$. Lemma \ref{lem3.2} is established in this case.

\medskip

Now we consider

\noindent{\bf Case two: $M_k^{\frac{2}{n-2}}q_{kn}\le C. $}

Let $v_k$ be defined as in (\ref{vkd1}).
In this case the boundary condition is written as
$$\partial_{y_n}v_k=(M_k^{-\frac{2}{n-2}}v_k)^{-\frac{n}{n-2}}h(M_k^{-\frac 2{n-2}}v_k)v_k^{\frac n{n-2}}, \quad
y_n=-M_k^{\frac 2{n-2}}q_{kn}. $$
$v_k$ converges to $V_1$ over all compact subsets of $\mathbb R^n\{y_n\ge -T\}$ where
$$T=\lim_{k\to \infty} M_k^{\frac{2}{n-2}}q_{kn}. $$
 $V_1$ satisfies (\ref{global2}).

For $R$ large and $\epsilon>0$ small, both independent of $k$, we have
$$v_k(y)\ge (\frac{n(n-2)}A-\epsilon)^{\frac{n-2}2}|y|^{2-n},\quad |y|=R. $$
In $B_R\cap \mathbb R^n_+$ we have the uniform convergence of $v_k$ to $V_1$. Our goal is to prove that $v_k$ is bounded below by $O(1)|y|^{2-n}$ outside $B_R$. To this end let
$$w(y)=(\frac{n(n-2)}A-2\epsilon)^{\frac{n-2}2}|y-A_1e_n|^{2-n} $$
where
$$A_1=c_h(\frac{n(n-2)}A)-T. $$
Then it is easy to check that
$$\frac{\partial w}{\partial y_n}>c_h w(y)^{\frac{n}{n-2}},\quad \mbox{ on } y_n=-M_k^{\frac{2}{n-2}}q_{kn}. $$
By choosing $R$ larger if needed we have
$$v_k(y)>(\frac{n(n-2)}A-\epsilon)^{\frac{n-2}2}|y|^{2-n}>w(y),\quad |y|=R,\quad y\in \mathbb R^n_+. $$
Then it is easy to apply maximum principle to prove $v_k>w$ in $\Omega_k\setminus B_R$.
Lemma \ref{lem3.2} is established.
\end{proof}

Let $q_1^k\in \Sigma_k$ and $q_2^k$ be its nearest or almost nearest sequence in $\Sigma_k$:
$$|q_2^k-q_1^k|=(1+o(1))d(q_1^k, \Sigma_k\setminus \{q_1^k\}). $$
 We claim that
\begin{lem}\label{compa-able}
There exists $C>0$ independent of $k$ such that
$$
\frac{1}{C}u_k(q_1^k)\leq u_k(q_2^k)\leq Cu_k(q_1^k).
$$
\end{lem}
\begin{proof}
Let $\sigma_k=d(q_1^k,\Sigma_k\setminus \{q_1^k\})$ and
$$\tilde u_k(y)=\sigma_k^{\frac{n-2}2}u_k(q_1^k+\sigma_ky). $$
We use $e_k$ to denote the image of $q_2^k$ after scaling (so $|e_k|\to 1$).
Then in $B_1$, $\tilde u_k(x)\sim \tilde u_k(0)^{-1}|x|^{2-n}$ for $|x|\sim 1/2$. On one hand, for $|x|=\frac 12$ we have, by Lemma \ref{lem3.2} applied to $e_k$,
$$\tilde u_k(0)^{-1}( \frac 12 )^{2-n}\ge C \tilde u_k(e_k)^{-1} $$
which is just $u_k(q_1^k)\le C u_k(q_2^k)$. On the other hand, the same moving sphere argument can be applied to $u_k$ near $q_2^k$ with no difference. The Harnack type inequality gives
$$\max_{B(q_2^k,1/4)\cap B_3^+}u_k\min_{B(q_2^k,1/2)\cap B_3^+}u_k\le C. $$
Using
$$\max_{B(q_2^k,1/4)\cap B_3^+}u_k\ge u_k(q_2^k)$$
 and
 $$\min_{B(q_2^k,1/2)\cap B_3^+}u_k\ge \min_{B(q_1^k,\sigma_k)\cap B_3^+}u_k$$
we have
\begin{equation}\label{aug27e1}
\tilde u_k(e_k)\tilde u_k(0)^{-1}\le C.
\end{equation}
Thus
(\ref{aug27e1}) gives
$u_k(q_2^k)\le C u_k(q_1^k)$. Lemma \ref{compa-able} is established.
\end{proof}

\begin{rem} Proposition \ref{prop3.2} is not needed in the proof of Lemma \ref{compa-able}.
\end{rem}

The following lemma is concerned with Pohozaev identity that can be verified by direct computation.

\begin{lem}\label{poho} Let $u$ solve
$$\left\{\begin{array}{ll}
\Delta u+ g(u)=0, \quad \mbox{ in } B_{\sigma}^+, \\
\\
\partial_n u=h(u)\quad \mbox{ on } \partial' B_{\sigma}^+.
\end{array}
\right.
$$
Then
\begin{align}\label{pigen}
& \int_{\partial' B_{\sigma}^+}h(u) (\sum_{i=1}^{n-1} x_i \partial_i u+\frac{n-2}2 u)+\int_{B_{\sigma}^+}(nG(u)-\frac{n-2}2 g(u)u)\\
=& \int_{\partial'' B_{\sigma}^+} \bigg (\sigma (G(u)-\frac 12 |\nabla u|^2+(\partial_{\nu} u)^2 )+\frac{n-2}2 u\partial_{\nu}u\bigg ) \nonumber
\end{align}
where $G(s)=\int_0^s g(t)dt$.
\end{lem}

\begin{prop} \label{sing-apart} There exists $d>0$ independent of $k$ such that
$$\lim_{k\to \infty}|q_1^k-q_2^k|\ge d. $$
\end{prop}
\begin{proof}

Recall that $\sigma_k=(1+o(1))|q_1^k-q_2^k|$. We prove by way of contradiction. Suppose $\sigma_k\to 0$,
et $\tilde M_k=\tilde u_k(0)$. We claim that
\begin{equation}\label{har-pos}
\tilde M_k \tilde u_k(y)\to a|y|^{2-n}+b(y) \mbox{ in } C^2_{loc}(B_{3/4}\cap \tilde \Omega_k\setminus \{0\}), \mbox{ with } a>0, \,\, b(0)>0
\end{equation}
where $\tilde \Omega_k=\{y;\quad \sigma_k y+q_1^k\in B_3^+\}$.

\medskip

\noindent{\bf Proof of (\ref{har-pos})}: As usual we consider the following two cases:

\noindent{\bf Case one: $\lim_{k\to \infty} q_{1n}^k \tilde M_k^{\frac 2{n-2}}\to \infty$}, and
{\bf Case two: $\lim_{k\to \infty} q_{1n}^k \tilde M_k^{\frac 2{n-2}}\to T<\infty$}

\medskip

First we consider {\bf Case one}. Let
$$T_k=\tilde M_k^{\frac 2{n-2}} q_{1n}^k. $$
Recall the equation for $\tilde u_k$ is (\ref{tildeuk}). Multiplying $\tilde M_k$ on both sides and letting $k\to \infty$ it is easy to see from the assumptions of $g$ and $h$ that
$\tilde M_k \tilde u_k\to h$ in $C^2_{loc}(B_1\setminus \{0\})$ where $h$ is a harmonic function defined in $B_1\setminus \{0\}$. Thus
 $$h(y)=a|y|^{2-n}+b(y)$$
 for some harmonic function $b(y)$ in $B_1$. From the pointwise estimate in Lemma \ref{lem3.2} we see that $a>0$.
Given any $\epsilon>0$, we compare $\tilde u_k$ and
$$w_k:=(a-\epsilon)(|y|^{2-n}-R_k^{2-n})$$ on $|y|\le R_k$. Here $R_k\to \infty$ is less than
$T_k$. Observe that
$\tilde u_k >w_k$ on $\partial B_{R_k}$ and $|y|=\epsilon_1$ for $\epsilon_1$ sufficiently small. Thus $\tilde u_k>w_k$ by the maximum principle. Let $k\to \infty$ we have, in $B_1$
$$a|y|^{2-n}+b(y)\ge (a-\epsilon)|y|^{2-n},\quad B_1\setminus B_{\epsilon_1}. $$
Then let $\epsilon\to 0$, which implies $\epsilon_1\to 0$ we have $b(y)\ge 0$ in $B_1$.
Next we claim that $b(0)>0$ because by Lemma \ref{lem3.2} and Lemma \ref{compa-able} we have
$$a|y|^{2-n}+b(y)\ge a_1|y-e|^{2-n} \quad \mbox{ in } B_1 $$
for some $a_1>0$, where $e=\lim_{k\to \infty} e_k$. Thus $b(y)>0$ when $y$ is close to $e$, which leads to $b(0)>0$.
(\ref{har-pos}) is established in {\bf Case one}.

\medskip

\noindent{\bf Case two}.

Again we first have $\tilde M_k \tilde u_k\to h$ in $C^2_{loc}(B_1^{-T}\setminus \{0\})$ and $h$ is of the form
$$h(y)=a |y|^{2-n}+b(y),\quad y_n\ge -T. $$
To prove $b(y)\ge 0$ we compare, for fixed $\epsilon>0$,
$\tilde M_k \tilde u_k$ with
$$w_k(y)=(a-\epsilon)(|y-b_k e_n |^{2-n}-(R_k-1)^{2-n}) $$
where $b_k\to 0$ and $R_k\to \infty$ are chosen to satisfy
$$(n-2)b_kR_k^{-2}>c_0 \sigma_k, \quad c_0=\sup_{0<s\le 1}s|h(s)| $$
and
$$(n-2)b_k> c_h \tilde M_k^{-\frac 2{n-2}}a^{\frac 2{n-2}}. $$
It is easy to see that such $b_k$ and $R_k$ can be found easily.
Let $h_k=\tilde M_k \tilde u_k$, and $\partial'\Omega_k=\partial \Omega_k\cap \{y_n=-T_k\}$. We divide $\partial'\Omega_k$ into two parts:
$$E_1=\{z\in \partial' \Omega_k;\quad \tilde u_k(z)\sigma_k^{-\frac{n-2}2}\ge 1 \,\, \},\quad E_2=\partial'\Omega_k\setminus E_1. $$
Then by the assumptions on $h$
$$\partial_n \tilde h_k\le \left\{\begin{array}{ll}
c_0\sigma_k h_k, \quad x\in E_2, \\
c_h\tilde M_k^{-\frac 2{n-2}} h_k^{\frac n{n-2}},\quad x\in E_1,
\end{array}
\right.
$$
With the choice of $b_k$ and $R_k$ it is easy to verify that
$$\partial_n w_k\ge \max\{ c_0\sigma_k w_k,c_h\tilde M_k^{-\frac 2{n-2}}w_k^{\frac{n}{n-2}}\} \mbox{ on } \partial' \Omega_k\cap B_{R_k}.$$
Thus standard maximum principle can be applied to prove $h_k\ge w_k$ on $\Omega_k\cap B_{R_k}$.
Letting $k\to \infty$ first and $\epsilon\to 0$ next we have $b(y)\ge 0$ in $B_1\cap \{y_n\ge -T\}$. Then by Proposition \ref{compa-able}
we see that $b(y)>0$ when $y$ is close to $e$, the limit of $e_k$. The fact $\partial_n b=0$ at $0$ implies $b(0)>0$. (\ref{har-pos}) is proved in both cases.

\medskip

Finally to finish the proof of Proposition \ref{sing-apart} we derive a contradiction from each of the following two cases:

\noindent{\bf Case one: $\lim_{k\to \infty} \tilde M_k^{\frac 2{n-2}} q_{1n}^k>0$. }

In this case we use the following Pohozaev identity on $B_{\sigma}$ for $\sigma<\lim_{k\to \infty} \tilde M_k^{\frac 2{n-2}}q_{1n}^k$:
\begin{align}\label{pohoz1}
&\int_{B_{\sigma}}\bigg (n G_k(\tilde u_k)-\frac{n-2}2\tilde u_k g_k(\tilde u_k) \bigg )\\
=& \int_{\partial B_{\sigma}}\bigg (\sigma( G_k(\tilde u_k)-\frac 12|\nabla \tilde u_k|^2+|\partial_{\nu} \tilde u_k|^2)+\frac{n-2}2\tilde u_k
\partial_{\nu}\tilde u_k\bigg )\nonumber
\end{align}
where
$$g_k(s)=\sigma_k^{\frac{n+2}2}g(\sigma_k^{-\frac{n-2}2}s), \quad G(t)=\int_0^tg(s)ds,\quad G_k(s)=\sigma_k^{n}G(\sigma_k^{-\frac{n-2}2}s). $$
First we claim that for $s>0$,
\begin{equation}\label{pilhs}
G_k(s)\ge \frac{n-2}{2n}s g_k(s).
\end{equation}
Indeed, writing $g(t)=c(t)t^{\frac{n+2}{n-2}}$, we see from $GH_1$ that $c(t)$ is a non-increasing function, thus
\begin{align*}
&G_k(s)=\sigma_k^n G(\sigma_k^{-\frac{n-2}2}s)=\sigma_k^n\int_0^{\sigma_k^{-\frac{n-2}2}s}c(t)t^{\frac{n+2}{n-2}}dt\\
&\ge  \sigma_k^nc(\sigma_k^{-\frac{n-2}2}s)\int_0^{\sigma_k^{-\frac{n-2}2}s}t^{\frac{n+2}{n-2}}dt=\frac{n-2}{2n}
c(\sigma_k^{-\frac{n-2}2}s)s^{\frac{2n}{n-2}}=\frac{n-2}{2n}s g_k(s).
\end{align*}

Replacing $s$ by $\tilde u_k$ we see that the left hand side of (\ref{pohoz1}) is non-negative.
Next we prove that
\begin{equation}\label{pohoe2}
\lim_{k\to \infty} \tilde M_k^2 \int_{\partial B_{\sigma}}\bigg (\sigma( G_k(\tilde u_k)-\frac 12|\nabla \tilde u_k|^2+|\partial_{\nu} \tilde u_k|^2)+\frac{n-2}2\tilde u_k
\partial_{\nu}\tilde u_k\bigg )<0
\end{equation}
for $\sigma>0$ small.  Clearly after (\ref{pohoe2}) is established we obtain a contradiction to (\ref{pohoz1}).
To this end first we prove that
\begin{equation}\label{sept5e1}
\tilde M_k^2G_k(\tilde u_k)=o(1).
\end{equation}
Indeed, by $GH_1$ and $GH_3$
\begin{align*}
G_k(\tilde u_k)&=\sigma_k^n \int_0^{\sigma_k^{-\frac{n-2}2}\tilde u_k}g(t)dt \\
\le &\left\{\begin{array}{ll}
\sigma_k^n\int_0^{\sigma_k^{-\frac{n-2}2}\tilde u_k} ctdt,\quad \mbox{ if } \sigma_k^{-\frac{n-2}2}\tilde M_k^{-1}\le 1, \\
\\
\sigma_k^n(\int_0^1 ctdt+\int_1^{\sigma_k^{-\frac{n-2}2}\tilde u_k}ct^{\frac{n+2}{n-2}}dt),\quad \mbox{ if } \sigma_k^{-\frac{n-2}2}\tilde M_k^{-1}
>1.
\end{array}
\right.
\end{align*}
Therefore
$$G_k(\tilde u_k)\le
 \left\{\begin{array}{ll} C\sigma_k^2\tilde M_k^{-2},\quad \mbox{ if } \sigma_k^{-\frac{n-2}2}\tilde M_k^{-1}\le 1, \\
C\sigma_k^n+C \tilde M_k^{-\frac{2n}{n-2}},\quad \mbox{ if } \sigma_k^{-\frac{n-2}2}\tilde M_k^{-1}>1.
\end{array}
\right.
$$
Clearly (\ref{sept5e1}) holds in either case. Consequently we write the left hand side of (\ref{pohoe2}) as
$$
\int_{\partial B_{\sigma}}(-\frac 12\sigma |\nabla h|^2+\sigma |\partial_{\nu} h|^2+\frac{n-2}2h \partial_{\nu}h)+o(1)
$$
where
$$h(y)=a|y|^{2-n}+b(y),\quad b(0)>0, a>0. $$
By direct computation we have
\begin{align*}
&\int_{\partial B_{\sigma}}(-\frac 12\sigma |\nabla h|^2+\sigma |\partial_{\nu} h|^2+\frac{n-2}2h \partial_{\nu}h)\\
=& \int_{\partial B_{\sigma}}\bigg (-\frac{(n-2)^2}a\cdot b(0) \cdot \sigma^{1-n}+O(\sigma^{2-n})\bigg )dS.
\end{align*}
Thus (\ref{pohoe2}) is verified when $\sigma>0$ is small.

The second case we consider is {\bf Case two: $\lim_{k\to \infty} \tilde M_k^{\frac{2}{n-2}}q_{1n}^k=0$}.

In this case we use the following Pohozaev identity on $B_{\sigma}^+$: Let
$$h_k(s)=\sigma_k^{\frac n2}h(\sigma_k^{-\frac{n-2}2}s),$$ then we have
\begin{align}\label{pohoz2}
& \int_{\partial B_{\sigma}^+\cap \partial \mathbb R^n_+}h_k(\tilde u_k) (\sum_{i=1}^{n-1} x_i \partial_i \tilde u_k+\frac{n-2}2 \tilde u_k)+\int_{B_{\sigma}^+}(nG_k(\tilde u_k)-\frac{n-2}2 g_k(\tilde u_k)\tilde u_k)\\
=& \int_{\partial B_{\sigma}^+\cap \mathbb R^n_+} \bigg (\sigma (G_k(\tilde u_k)-\frac 12 |\nabla \tilde u_k|^2+(\partial_{\nu} \tilde u_k)^2 )+\frac{n-2}2 \tilde u_k\partial_{\nu}\tilde u_k\bigg ) \nonumber
\end{align}
Multiplying $\tilde M_k^2$ on both sides and letting $k\to \infty$ we see by
the same estimate as in {\bf Case one} that the second term on the left hand side is non-negative, the right hand side is strictly negative. The only term we need to consider is
$$\lim_{k\to \infty} \tilde M_k^2 \int_{\partial B_{\sigma}^+\cap \partial \mathbb R^n_+}h_k(\tilde u_k) (\sum_{i=1}^{n-1} x_i \partial_i \tilde u_k+\frac{n-2}2 \tilde u_k). $$
Let $H(s)=\int_0^sh(t)dt$, then from integration by parts we have
\begin{align}\label{sept5e3}
&\int_{\partial B_{\sigma}^+\cap \partial \mathbb R^n_+}h_k(\tilde u_k) (\sum_{i=1}^{n-1} x_i \partial_i \tilde u_k+\frac{n-2}2 \tilde u_k)\\
=&\int_{\partial B_{\sigma}\cap \partial \mathbb R^n_+}\sigma_k^{n-1}H(\sigma_k^{-\frac{n-2}2}\tilde u_k)\sigma \nonumber\\
& +\int_{\partial B_{\sigma}^+\cap \partial \mathbb R^n_+}(-(n-1)\sigma_k^{n-1}H(\sigma_k^{-\frac{n-2}2}\tilde u_k)+\frac{n-2}2\tilde u_k
h_k(\tilde u_k))dx'.
\nonumber
\end{align}
For the first term on the right hand side of (\ref{sept5e3}) we claim
\begin{equation}\label{sept5e5}
\tilde M_k^2 \sigma_k^{n-1}H(\sigma_k^{-\frac{n-2}2}\tilde u_k)=o(1) \quad \mbox{ on }\partial B_{\sigma}.
\end{equation}
Indeed, by $GH_1$ and $GH_2$
$$| H(\sigma_k^{-\frac{n-2}2}\tilde u_k)|
\le \left\{\begin{array}{ll}
\int_0^{\sigma_k^{-\frac{n-2}2}\tilde u_k}ctdt, \quad\mbox{ if } \sigma_k^{-\frac{n-2}2}\tilde u_k\le 1, \\
\\
\int_0^1 ct dt +\int_1^{\sigma_k^{-\frac{n-2}2}\tilde u_k} ct^{\frac n{n-2}}dt, \quad \mbox{ if }  \sigma_k^{-\frac{n-2}2}\tilde u_k>1,
\end{array}
\right.
$$
Using $\tilde u_k=O(1/\tilde M_k)$ on $\partial B_{\sigma}$ we then have
$$\tilde M_k^2 \sigma_k^{n-1} | H(\sigma_k^{-\frac{n-2}2}\tilde u_k)|
\le \left\{\begin{array}{ll}
O(\sigma_k),\quad \mbox{ if } \sigma_k^{-\frac{n-2}2}\tilde u_k\le 1, \\
O(\sigma_k)+O(\tilde M_k^{-\frac{2}{n-2}}),\quad \mbox{ if }  \sigma_k^{-\frac{n-2}2}\tilde u_k>1.
\end{array}
\right.
$$
Thus (\ref{sept5e5}) is verified and the first term on the right hand side of (\ref{sept5e3}) is $o(1)$.

Therefore we only need to estimate the last term of (\ref{sept5e3}), which we claim is non-negative. Indeed, for $t>0$, we write $h(t)=b(t)t^{\frac{n}{n-2}}$ for some
non-decreasing function $b$. Then we have
\begin{align*}
&\sigma_k^{n-1}H(\sigma_k^{-\frac{n-2}2}s)
= \sigma_k^{n-1}\int_0^{\sigma_k^{-\frac{n-2}2}s}h(t)dt\\
=& \sigma_k^{n-1}\int_0^{\sigma_k^{-\frac{n-2}2}s} b(t) t^{\frac {n}{n-2}} dt
\le  \sigma_k^{n-1} b(\sigma_k^{-\frac{n-2}2}s)\int_0^{\sigma_k^{-\frac{n-2}2}s}t^{\frac{n}{n-2}}dt\\
=& \frac{n-2}{2n-2}b(\sigma_k^{-\frac{n-2}2}s) s^{\frac{2n-2}{n-2}}=\frac{n-2}{2n-2} h_k(s) s.
\end{align*}
Replacing $s$ by $\tilde u_k$ in the above we see that the last term of (\ref{sept5e3}) is non-negative. Thus there is a contradiction in (\ref{pohoz2}) in {\bf Case two} as well.  Proposition \ref{sing-apart} is established.
\end{proof}

\medskip

We are in the position to finish the proof of Theorem \ref{mainthm1}. By Proposition \ref{sing-apart} there is a positive distance between any two members of $\Sigma_k$. The uniform bound of $\int_{B_1^+}u_k^{2n/n-2}$ follows readily from the pointwise estimates for blowup solutions near an isolated blowup point. The uniform bound for $\int_{B_1^+} |\nabla u_k|^2$ can be obtained by scaling and standard elliptic estimates for linear equations. Thus we have obtained a contradiction to (\ref{aug14e1}). Theorem \ref{mainthm1} is established. $\Box$

\section{Proof of Theorem \ref{minorthm}}

If $h$ is non-positive, the energy estimate follows from the Harnack inequality in a straight forward way. Indeed, let $G(x,y)$ be
a Green's function on $B_3^+$ such that $G(x,y)=0$ if $x\in B_3^+,y\in \partial B_3^+\cap \mathbb R^n_+$ and $\partial_{y_n} G(x,y)=0$ for $x\in B_3^+,y\in
\partial B_3^+\cap \partial \mathbb R^n_+$. It is easy to see that $G$ can be constructed by adding the standard Green's function on $B_3$ its reflection over $\partial \mathbb R^n_+$. It is also immediate to observe that
$$G(x,y)\ge C_n |x-y|^{2-n}, \quad x\in B^+_3,\quad y\in B_2^+. $$
Multiplying $G$ on both sides of (\ref{equ23}) and integrating by parts, we have
\begin{align*}
u(x)+\int_{\partial B_3^+\cap \partial \mathbb R^n_+} h(u(y))G(x,y)dS_y+\int_{\partial B_3^+\cap \mathbb R^n_+}u(y)\frac{\partial G(x,y)}{\partial \nu}dS_y \\
=\int_{B_3^+}g(u(y))G(x,y)dy.
\end{align*}
Here $\nu$ represents the outer normal vector of the domain. Using $h\le 0$ and $\partial_{\nu}G\le 0$, we have
$$u(x)\ge \int_{B_3^+}g(u(y))G(x,y)dy, \quad x\in B_3^+. $$
In particular take let $u(x_0)=\min_{\partial B_2^+}u$, then $|x_0|=2$, thus
$$C\ge \max_{B_1^+}u\cdot \min_{B_2^+}u\ge \int_{B_{3/2}^+}g(u(y))u(y)G(x_0,y)dy\ge C\int_{B_{3/2}^+}g(u)udy. $$
Therefore we have obtained the bound on $\int_{B_{3/2}^+}g(u)udy$. To obtain the bound on $\int_{B_1^+} |\nabla u|^2$, we use a cut-off function $\eta$ which is $1$ on $B_1^+$ and is $0$ on $B_2^+\setminus B_{3/2}^+$ and $|\nabla \eta |\le C$. Multiplying $u\eta^2$ to both sides of (\ref{equ23}) and using integration by parts and Cauchy inequality we obtain the desired bound on $\int_{B_1^+}|\nabla u|^2$.
Theorem \ref{minorthm} is established. $\Box$


\begin{thebibliography}{99}

\bibitem{caffarelli} Caffarelli, Luis A.; Gidas, Basilis; Spruck, Joel Asymptotic symmetry and local behavior of semilinear elliptic equations with critical Sobolev growth. Comm. Pure Appl. Math. 42 (1989), no. 3, 271-297.

\bibitem{chen-lin-cpam1} Chen, Chiun Chuan; Lin, Chang Shou, Estimates of the conformal scalar curvature equation via the method of moving planes. Comm. Pure Appl. Math. 50 (1997), no. 10, 971-1017.

\bibitem{chen-lin-2} Chen, Chiun Chuan; Lin, Chang Shou,
Local behavior of singular positive solutions of semilinear elliptic equations with Sobolev exponent.
Duke Math. J. 78 (1995), no. 2, 315-334.

\bibitem{dmo} Djadli, Zindine; Malchiodi, Andrea; Ould Ahmedou, Mohameden,
Prescribing scalar and boundary mean curvature on the three dimensional half sphere.
J. Geom. Anal. 13 (2003), no. 2, 255–289.

\bibitem{gui-lin} Gui, Changfeng; Lin, Chang-Shou,
Estimates for boundary-bubbling solutions to an elliptic Neumann problem.
J. Reine Angew. Math. 546 (2002), 201–235.

\bibitem{hanli}  Han, Zheng-Chao; Li, Yanyan The Yamabe problem on manifolds with boundary: existence and compactness results. Duke Math. J. 99 (1999), no. 3, 489–542.

\bibitem{aobing} Li, Aobing; Li, Yanyan On some conformally invariant fully nonlinear equations. Comm. Pure Appl. Math. 56 (2003), no. 10, 1416--1464.

\bibitem{yanyanli} Li, YanYan, Prescribing Scalar Curvature on $\mathbb S^n$ and Related Problems, Part I, J. Differential Equations 120 (1995), no. 2, 319-410.


\bibitem{lei}  Li, YanYan; Zhang, Lei Liouville-type theorems and Harnack-type inequalities for semilinear elliptic equations. J. Anal. Math. 90 (2003), 27-87.

\bibitem{li-zhang-4d}  Li, Yan Yan; Zhang, Lei A Harnack type inequality for the Yamabe equation in low dimensions. Calc. Var. Partial Differential Equations 20 (2004), no. 2, 133-151.

\bibitem{liandzhu}  Li, YanYan; Zhu, Meijun Uniqueness theorems through the method of moving spheres. Duke Math. J. 80 (1995), no. 2, 383-417.

\bibitem{lin-pra}  Lin, Chang-Shou; Prajapat, Jyotshana V. Harnack type inequality and a priori estimates for solutions of a class of semilinear elliptic equations. J. Differential Equations 244 (2008), no. 3, 649–695.

\bibitem{schoen1} Schoen, Richard; Lecture notes of Stanford University.

\bibitem{schoen-zhang} Schoen, Richard; Zhang, Dong;
Prescribed scalar curvature on the n-sphere.
Calc. Var. Partial Differential Equations 4 (1996), no. 1, 1-25.

\bibitem{zhang1} Zhang, Lei, Classification of conformal metrics on $\mathbb R^2_+$ with constant Gauss curvature and geodesic curvature on the boundary under various integral finiteness assumptions. Calc. Var. 16, 405-430 (2003).

\bibitem{zhang2} Zhang, Lei Harnack type inequalities for conformal scalar curvature equation. Math. Ann. 339 (2007), no. 1, 195–220.
\end{thebibliography}
\end{document}